\documentclass[11pt,a4paper]{article}
\usepackage[utf8]{inputenc}
\usepackage[OT4]{fontenc}
\usepackage{latexsym}
\usepackage{amsfonts}
\usepackage{amsmath}
\usepackage{amssymb}
\usepackage{amsthm}
\usepackage{extpfeil}
\usepackage{verbatim}
\usepackage[nobysame]{amsrefs}
\usepackage{hyperref}
\hypersetup{pdfborder=false, pdftitle={Remarks on coarse triviality of asymptotic Assouad-Nagata dimension}}

\newtheorem{thm}{Theorem}[section]
\newtheorem*{thm*}{Theorem}
\newtheorem{lem}[thm]{Lemma}
\newtheorem{rem}[thm]{Remark}
\newtheorem{prop}[thm]{Proposition}
\newtheorem{cor}[thm]{Corollary}
\newtheorem{ex}[thm]{Example}
\theoremstyle{definition}
\newtheorem{defi}[thm]{Definition}

\DeclareMathOperator{\asdim}{asdim}
\DeclareMathOperator{\asdimAN}{asdim_{AN}}
\DeclareMathOperator{\im}{im}
\def\U{{\mathcal U}}
\def\R{{\mathbb R}}
\def\N{{\mathbb N}}
\def\Z{{\mathbb Z}}
\def\notion{\textit}
\def\eps{\varepsilon}
\def\implied{\Longleftarrow}

\author{Damian Sawicki\thanks{damian.sawicki@students.mimuw.edu.pl}}

\title{Remarks on coarse triviality\\
of asymptotic Assouad-Nagata dimension}

\date{September 26, 2013}

\begin{document}

\maketitle

\begin{abstract}
We show for a given metric space $(X,d)$ of asymptotic dimension~$n$ that there exists a coarsely and topologically equivalent hyperbolic metric~$d'$ of the form $d'=f\circ d$ such that $(X,d')$ is of asymptotic Assouad-Nagata dimension~$n$. As a corollary we construct examples of spaces realising strict inequality in the logarithmic law for $\asdimAN$ of a Cartesian product. One of them may be viewed as a counterexample to a specific kind of a Morita-type theorem for $\asdimAN$.
\end{abstract}

\section*{Introduction}
In large scale geometry two categories are usually considered: the coarse category and the large scale Lipschitz category. Since any large scale Lipschitz isomorphism (\notion{quasi-isometry}) is a coarse isomorphism, but not conversely, one may ask about the differences between those categories -- for example how may large scale Lipschitz invariants differ in a class of coarsely equivalent spaces and what relations are there between invariants of both categories.

The key tool is the principle that a metric composed with a concave function is still a valid metric; i.e., it satisfies the triangle inequality. This leads to the main result:

\begin{thm}\label{mainTheorem}
For any metric space $(X,d)$ of finite asymptotic dimension $\asdim (X,d) = n$ there is a coarsely and topologically equivalent hyperbolic metric~$d'$ of the form $d'=f\circ d$, such that asymptotic Assouad-Nagata dimension $\asdimAN(X,d')$ is also equal to $n$.
\end{thm}

One should note that core of the above result is not new. A similar theorem was proved by Brodskiy et al. in \cite[thm. 5.1]{viaLipschitz} and is also a~consequence of the fact that universal spaces for asymptotic dimension of Dranishnikov and Zarichnyi allow linear control (\cite[prop. 4.1]{Universal}, for proper metric spaces). In \cite[prop. 6.6]{BDLM} for countable groups  Brodskiy et al. found an equivalent metric which is additionally left-invariant.

The new is a more abstract point of view relying on the notion of the \emph{control function}\footnote{Language of $n$-dimensional control functions (see def. \ref{controlFunction}) proposed for example in~\cite{BDLM} is very useful for comparing different definitions of dimension.}  for describing asymptotic dimension and on construction of the new metric directly from the original. This approach allows to save much of the structure of the original metric. In particular, left-invariance for groups is preserved as in \cite{BDLM} and if the original metric is an $\ell_\infty$-metric on a Cartesian product then so is the new one. Due to the latter one can immediately construct examples realising a strict inequality in the logarithmic law: $\asdimAN X\times Y \leq \asdimAN X + \asdimAN Y$, from the known examples for $\asdim$ (\cite[corollary 7.7]{firstCounter}, \cite[theorem 5.3]{counterMorita}, \cite[section 5]{Lebedeva}). Earlier examples with such properties can be found in \cite{firstCounter} and \cite{Lebedeva}.

\subsection*{Acknowledgements}
This paper is based on a BSc thesis written at the Faculty of Mathematics, Informatics and Mechanics of the University of Warsaw (\cite{licencjat}).

The author is grateful to the supervisor of the original BSc thesis,\linebreak dr~Tadeusz~Koźniewski and the co-host of the BSc seminar, prof. Sławomir Nowak. He received valuable help from them on preparing both the thesis and this paper. I would also like to thank dr Piotr Nowak for pointing out \cite{Lebedeva} as a source of examples for $\asdimAN$, proofreading and a lot of valuable corrections.


\section{Preliminaries}

\subsection*{Definitions}

\begin{defi}\label{controlFunction}
For a metric space $X$ an \emph{$n$-dimensional control function} is any function $D_X\colon \R_+\to \R_+\cup \{\infty\}$ such that for every $r \in \R_+$ there is a cover $\U=\bigcup\limits_{i=1}^{n+1}\U_i$ of $X$ such that
each family $\U_i$ is $r$-disjoint and the diameter of elements of $\U$ is bounded by $D_X(r)$.
\end{defi}

\begin{defi}
The space $X$ is of \emph{asymptotic dimension} at most $n$, denoted $\asdim X \leq n$, if there is a \emph{real-valued} $n$-dimensional control function for $X$.
\end{defi}
\begin{defi}
The space $X$ is of \emph{asymptotic Assouad-Nagata dimension} at most $n$, denoted $\asdimAN X \leq n$ if there is an \emph{affine} $n$-dimensional control function for $X$.
\end{defi}

\subsection*{Main tool}

\begin{prop}\label{prelemma} Assume $c:[0,\infty)\to[0,\infty)$ is nondecreasing. The following conditions are equivalent:
\begin{itemize}
\item[--]
$c$ is subadditive (i.e., $c(a+b)\leq c(a)+c(b)$) and $c(r)=0$ iff $r=0$,
\item[--]
for any metric space $(X,d)$ function $d'=c\circ d$ is a valid metric on $X$.
\end{itemize}
\end{prop}

\begin{proof} $\implies$ It is obvious that $d'$ is symmetric and $d'(x,y) = 0$ iff $x=y$. For $x,y,z\in X$ we obtain the triangle inequality as follows:
$$
c\circ d(z,x) \leq c(d(x,y)+d(y,z)) \leq c\circ d(x,y) + c\circ d(y,z),
$$
where the first inequality follows from the monotonicity of $c$ and the triangle inequality for $d$ and the second from the subadditivity of $c$.

$\implied$ The second condition is obvious. For the first consider $X=\mathbb R$ (with the standard metric induced by the absolute value) and $a,b\geq 0$. Then from the triangle inequality $d'(a+b,0)\leq d'(0,a)+d'(a,a+b)$ we get subadditivity\footnotemark\ $c(a+b)\leq{}c(a)+c(b)$.
\end{proof}

\footnotetext{Note that we did not use the extra assumption that $c$ is nondecreasing for this implication. In fact it is possible that $c$ is not nondecreasing if we only assume that $c\circ d$ is a metric for any $d$: take $0<p\leq q \leq r \leq 2p$. Then for any $c:\R_+\to\{p,q,r\}$ and any $d$ the composition $c\circ d$ is a valid metric.}

\begin{lem}\label{lemma} Assume $c:[0,\infty)\to[0,\infty)$ is concave and $c(r)=0$ iff $r=0$. Then for any metric space $(X,d)$ function $d'=c\circ d$ is also a valid metric on $X$.
\end{lem}
\begin{proof} We will use the previous proposition. It is easy to notice that a nonnegative concave function on $[0,\infty)$ must be nondecreasing, so it is enough to check that $c$ is subadditive. Take any $a,b>0$:
\begin{multline*}
c(a+b) =
a\cdot \frac{c(a+b)}{a+b} + b \cdot \frac{c(a+b)}{a+b}
\leq
a\cdot \frac{c(a)}{a} + b \cdot \frac{c(b)}{b} = c(a)+c(b)
\end{multline*}
where the inequality is a consequence of the fact that concave functions have monotone difference quotients.
\end{proof}

There are functions satisfying assumptions of proposition \ref{prelemma} that are not concave ($c(r) = \max(\min(r,\ 1/2),\ r/2)$), do not have monotone difference quotiens of the form $\frac{c(r)}{r}$ ($c(r)=\frac{\lfloor r\rfloor}{2} + \min(\{r\}, 1/2)$) or even have discrete codomain ($c(r) = \lceil \log_2r\rceil_+ + 1$), but we will use concave functions only.

\begin{rem}\label{equivalentMetric}
Under the assumptions of the previous lemma: if $c$ is unbounded, then $(X,d)$ and $(X,d')$ are coarsely equivalent (via the identity function) and if it is continuous at $0$, then they are homeomorphic.
\end{rem}
\begin{proof} First assume that $c$ is unbounded. It follows that it is strictly increasing and thus has an inverse $c^{-1}$. Thus we have a two-sided inequality (which is equality in fact) from the definition of coarse equivalence for $id_X^{-1}: (X,d') \to (X,d)$:
$$ c^{-1}\big(d'(x,y)\big) \leq d(x,y) \leq c^{-1}\big(d'(x,y)\big),$$
where $c^{-1}$ is monotone and unbounded.

If $c$ is continuous at $0$, then $id_X$ is clearly continuous, and since $c$ must be strictly increasing in some neighbourhood of $0$, then locally $c$ has an inverse $c^{-1}$ which is continuous at $0$ and thus $id_X^{-1}$ is continuous, too.
\end{proof}

\begin{rem}\label{nonIncreasingControl}
Under the assumptions of lemma \ref{lemma}: $\asdimAN (X,d')$ is at~most $\asdimAN (X,d)$. More precisely, if $D_X$ is an affine $n$-dimensional control function for $(X,d)$, then then there exists a constant $C$ such that $D_X'(r)=D_X(r)+C$ is an $n$-dimensional control function for $(X,d')$.
\end{rem}
\begin{proof}
The bounded case is trivial, so one can assume that $c$ is unbounded. As such, it is strictly increasing and, consequently, has an inverse with domain $\im(c)=\{0\}\cup(a, \infty)$. Let us extend this inverse to $[0,\infty)$ by $0$ and denote by~$c^{-1}$.

Let $D_X(r)=Ar + B$, with $A\geq 1$. Clearly $\tilde D_X(r') = c \circ D_X(c^{-1}(r'))$ is a control function for $(X,d')$. For simplicity we will denote $c^{-1}(r')$ by $r$ and assume it is nonzero.
\begin{multline*} \tilde D_X(r') = c (A\cdot r + B) = \big(c (A\cdot r) - c(0)\big) + \big( c(A\cdot r+B) - c(A\cdot r) \big) = \\
=
A\cdot r \cdot \frac{c (A\cdot r) - c(0)}{A\cdot r} +
B\cdot\frac{c(A\cdot r + B) - c(A\cdot r)}{B} \leq \\
\leq
A\cdot r \cdot \frac{c (r) - c(0)}{r} +
B\cdot\frac{c(B) - c(0)}{B} = A r' + c(B),
\end{multline*}
where the inequality relies on the concavity of $c$.
\end{proof}


\section[Coarse triviality of asdimAN]{Coarse triviality of $\boldsymbol\asdimAN$}
Let us recall and prove the main theorem \ref{mainTheorem} stated already in the Introduction.

\begin{thm*} Assume $\asdim (X,d) = n$. There is a hyperbolic metric $d'$, coarsely and topologically equivalent to $d$, of the form $d'=c\circ d$, such that $\asdimAN(X,d') = \asdim(X,d') =n$.
\end{thm*}
\begin{proof}
We will use lemma \ref{lemma} together with the subsequent remark \ref{equivalentMetric}. We will inductively construct an appropriate function~$c$. Let $D_X$ be a real-valued $n$-dimensional control function for $X$. Without loss of generality one can assume that $D_X$ is nondecreasing.

Let $c_{|[0,1]}$ be the identity function, $a_{-1}=0$ and $a_0=1$. In the $k$th inductive step we will assume that $c$ is defined for $[0,a_{k-1}]$ as a piecewise affine, increasing concave function and $c(a_{l-1})=l$ for all $l\in\N\cap[0,k]$. So, let
\begin{equation}\label{a_kFormula}
a_k=\max\big( D_X(a_{k-1}),\ \ a_{k-1} + (a_{k-1} - a_{k-2}) \big)
\end{equation}
-- the role of the first argument of $\max$ is to decrease $\asdimAN$ and the second argument guarantees concavity of $c$. Set $c(a_k)=k+1$ and let $c$ be affine on $[a_{k-1},a_k]$.

The final function $c$ clearly satisfies the assumptions of \ref{lemma} and \ref{equivalentMetric}. We will show that $\asdimAN(X,d')=n$ for $d'=c\circ d$. As before (\ref{nonIncreasingControl}) we have a control function $\tilde D_X$ for $(X,d')$ of the form: $\tilde D_X(r') = c\circ D_X(c^{-1}(r'))$. Assume that $r'\in[k-1,k]$ for some $k\in \N$; i.e., that $c^{-1}(r')\in [a_{k-2}, a_{k-1}]$. Then, by the monotonicity of $D_X$ and formula \ref{a_kFormula},
$$D_X(c^{-1}(r'))\leq D_X(a_{k-1})\leq~a_k,$$
so:
$$\tilde D_X(r') = c\big(D_X(c^{-1}(r'))\big) \leq c(a_k) = k+1 \leq r' + 2.$$

It remains to handle hyperbolicity. It is known from the classic works of Gromov that for any metric space $(Y,d_Y)$ and $l(r)=\ln(r+1)$, the metric space $(Y, l\circ d_Y)$ is hyperbolic. By remark \ref{nonIncreasingControl} we know that for the metric $d''=l \circ c\circ d$ it is still true that $\asdimAN(X,d'')=n$ (with control function $D_{(X,d'')}(r'') = r'' + \ln3$).
\end{proof}

An interested reader may compare the above proof with the proof in~\cite{viaLipschitz}.

The fact that $d'$ from the above theorem is a function of the original metric $d$ implies that a lot of properties of $(X,d)$ is inherited by $(X,d')$, for example:
left-invariance of metric;
product structure (if $X=\prod X_i$ and $d=\sup d_i$, then $d' = \sup d_i'$);
isometric group actions on $X$;
properness;
topology.
What can't be inherited is the metric being a word metric or even being only quasi-geodesic, because any coarse equivalence of quasi-geodesic spaces is a quasi-isometry and quasi-isometries preserve $\asdimAN$.

In the subsequent corollary and examples we will omit the information that the respective metrics may be required to be hyperbolic.

\begin{cor}\label{mainForProduct} Assume $X=\prod_{i=1}^l X_i$, $d(x,y)=\max d_i(x_i,y_i)$ and $\asdim$ of $(X,d), (X_i,d_i)$ is $n,n_i$ respectively. There is a function $c$, such that metrics~$d',d_i'$ of the form $d'=c\circ d$, $d_i'=c \circ d_i$  are coarsely and topologically equivalent to the original and the following equalities hold:
\begin{alignat*}{3}
\asdimAN(X,d')     & = & \asdim(X,d')    & = n\\
\asdimAN(X_i,d_i') & = &\ \asdim(X_i,d_i') & = n_i.
\end{alignat*}
\end{cor}
\begin{proof} We use the same idea as in the proof of the previous theorem. One should correct the formula \ref{a_kFormula} to the following:
$$a_k=\max\Big(
D_X(a_{k-1}),\ \
\max_i D_{X_i}(a_{k-1}),\ \
a_{k-1} + (a_{k-1} - a_{k-2})
\Big),$$
where $D_{X_i}$ is an $n_i$-dimensional control function for $X_i$.
\end{proof}

Constructions similar to the last corollary are possible: for example for an exact sequence of groups instead of the Cartesian product.


\section{Applications}

We will show how to use the derived results to construct examples of spaces realising strict inequality in the logarithimic formula for $\asdimAN$ (see \cite[theorem 2.5]{BDLM}):
$$\asdimAN X_1\times X_2 \leq \asdimAN X_1 + \asdimAN X_2.$$

\begin{thm}
There are metric spaces $X_1, X_2$ satisfying
$$\asdimAN X_1\times X_2 < \asdimAN X_1 + \asdimAN X_2.$$
\end{thm}
\begin{proof} It is enough to take an example of spaces $(X_1,d_1), (X_2,d_2)$ satisfying $\asdim X_1\times X_2 < \asdim X_1 + \asdim X_2$ and apply corollary \ref{mainForProduct}, to get $(X_1,d_1'),(X_2,d_2')$ meeting the required conditions.
\end{proof}

The classic Morita theorem for dimension states that $\dim X\times \R$ is equal to $\dim X+1$. If we start with a counterexample to the Morita-type theorem for~$\asdim$, we obtain the following\footnote{A counterexample for $\asdim$ with the desired properties can be found in \cite[theorem~5.3]{counterMorita}, the asymptotic dimension of $X$ is equal to $2$.}:

\begin{ex} There are metric spaces $X,R$, where $R$ is coarsely equivalent to the reals $\R$ by a change of metric as in \ref{mainTheorem} (thus by \ref{nonIncreasingControl} $\asdimAN R=~\!\!1$) and $X$ is proper, of bounded geometry, such that $$\asdimAN X\times R = \asdimAN X < \asdimAN X + \asdimAN R.$$
\end{ex}

Interestingly, the above example shows that a Morita-type theorem for $\asdimAN$ (\cite[theorem 4.3]{ANdimension}) requires strong assumptions to be true.


\section{Final remark on quasi-isometries}

\begin{prop} For a pair of metric spaces $(X,d_X), (Y,d_Y)$ and a coarse equivalence $f\colon (X,d_X)\to (Y,d_Y)$, there exist coarsely and topologically equivalent hyperbolic metrics $d_X',d_Y'$ of the form $d_X'=c_X\circ d_X$, $d_Y'=c_Y\circ~\!d_Y$ such that $f\colon (X,d_X')\to (Y,d_Y')$ is a quasi-isometry. Moreover, it is an isometry up to an additive constant that may be chosen to be arbitrarily small.
\end{prop}

\begin{proof} Let $\phi,\Phi$ be the contraction and dilation functions\footnote{They are monotone and unbounded functions from the definition of coarse equivalence.} of $f$:
$$\phi\big(d_X(x,x')\big) \leq d_Y\big(f(x), f(x')\big) \leq \Phi\big(d_X(x,x')\big).$$

We will construct $c_X$ and $c_Y$ in inductive steps as in the proof of \ref{mainTheorem} and the claim will follow from \ref{lemma} and \ref{equivalentMetric}. The fact that $f$ is a quasi-isometry will be checked directly.

Let $a_{-1}=0$, $a_0=1$, $c_X(0)=0$, $c_X(1)=1$ and $c_X$ be affine on the interval $[0,1]$. Moreover let $b_{-1}=0$ and $c_Y(0)=0$. Let us enumerate inductive steps by $\{(0,1)\} \cup \N_+\times\Z_2$ with the lexical order.

In the step number $(0,1)$ let $b_0 = \max\big(\Phi(a_0),1\big)$, $c_Y(b_0)=1$ and $c_Y$ be affine on $[0,b_0]$.

In the step number $(k,0)$ let
$$a_k = 
\max\big(
\sup \phi^{-1}([0,b_{k-1}]),\ \
a_{k-1} + (a_{k-1} - a_{k-2})
\big),$$
$c_X(a_k)=k+1$ and $c_X$ be affine on the interval $[a_{k-1}, a_k]$.

In the step number $(k,1)$ let
$$b_k =
\max\big(
\Phi(a_k),\ \
b_{k-1} + (b_{k-1} - b_{k-2})
\big),$$
$c_Y(b_k)=k+1$ and $c_Y$ be affine on the interval $[b_{k-1}, b_k]$.

Now we have (indices of metrics are skipped for simplicity):
\begin{eqnarray*}
\phi\big(d(x,x')\big) \leq & d\big(f(x), f(x')\big) & \leq \Phi\big(d(x,x')\big)\\
c_Y \circ \phi\big( d(x,x')\big) \leq & d'\big(f(x), f(x')\big) &  \leq c_Y \circ \Phi\big( d(x,x')\big).
\end{eqnarray*}
Let $x,x'\in X$ and $d(x,x')=r$, $d'(x,x')=r'$. Assume that $r\in [a_{k-1},a_k]$ for some $k\in \N$ (it means that $r'\in[k,k+1]$). Then, by the formula for $a_{k-1}$ we have $\phi(r) \geq b_{k-2}$, and by the formula for $b_k$: $\Phi(r)\leq b_k$. Thus the following inequalities hold:
\begin{eqnarray*}
k-1 \leq c_Y \circ \phi(d(x,x')) \leq & d'\big(f(x), f(x')\big) & \leq c_Y \circ \Phi(d(x,x'))\leq k+1 \\
d'(x,x') - 2 \leq & d'\big(f(x), f(x')\big) & \leq d'(x,x') + 1.
\end{eqnarray*}

A simple calculation shows that after correcting both $c_X, c_Y$ by composition with the function $l$ given by $l(r)=\ln(r+1)$ to achieve hyperbolicity, the last double inequality still holds\footnote{In fact composition with $l$ translates any quasi-isometry to what we called ``an isometry up to a constant''.}. In order to obtain smaller additive constants it suffices to further multiply both metrics by a small constant~$\eps$.
\end{proof}

A similar argument justifies the following:

\begin{prop} For a pair of metric spaces $(X,d_X), (Y,d_Y)$ and a coarse map $f\colon X\to Y$, there exists a coarsely and topologically equivalent hyperbolic metric $d_Y'$ of the form $d_Y'=c_Y\circ d_Y$ such that $f\colon (X,d_X)\to (Y,d_Y')$ is large scale Lipschitz.
\end{prop}

\end{document}